\documentclass[10pt]{amsart}
\usepackage{amssymb,amsmath,euscript}
\usepackage{mathtools}
\usepackage{breqn}

\DeclarePairedDelimiter\floor{\lfloor}{\rfloor}
\usepackage{}

\newtheorem{theorem}{Theorem}[section]
\newtheorem{lemma}[theorem]{Lemma}

\newtheorem{corollary}[theorem]{Corollary}
\newtheoremstyle{case}{}{}{}{}{}{:}{ }{}
\theoremstyle{case}
\newtheorem{case}{Case}
\theoremstyle{definition}

\theoremstyle{remark}
\newtheorem{remark}[theorem]{Remark}
\def\Fq{{\mathbb F}_q}

\def\FF{{\mathbb F}}
\def\PP{{\mathbb P}}

\newcommand{\tr}{\operatorname{tr}}

\newcommand{\Ev}{\operatorname{Ev}}

\newcommand{\w}{\operatorname{w}_H}
\newcommand{\rank}{\operatorname{rank}}

\newcommand{\supp}{\mathrm{S}}

\def\Pam{\mathbb{P}^{{m\choose 2}-1}}
\def\X{\mathbf{X}}
\def\A{A^\prime}
\def\Ak{A^\prime_{2k-1}}

\def\Am{\mathbb{A}_m}
\def\DetAm{\mathbb{A}(t, m)}
\def\DetAmt{\mathbb{A}(2t, m)}

\def\Amr{A(2r, m)}

\def\narm{n_a(2r, m)}
\def\Natm{N_a(2t, m)}

\def\CAtm {C_{\mathbf{A}}(2t, m)}
\def\C {\widehat{C}_{\mathbf{A}}(2t, m)}

\def\Wk {\mathrm{W}_{2k}(2t, m)}
\def\Wkr {\mathrm{w}_{2k}(2r, m)}

\def\DetVar{\mathbf{Det_A}(2t, m)}

\def \Lat{\Lambda_a(2t, m)}

\begin{document}
	
	\title[Codes Associated to Skew-symmetric Determinantal Varieties]{Linear Codes Associated to Skew-symmetric Determinantal Varieties}
	\author{Peter Beelen }
\address{Department of Applied Mathematics and Computer Science,\newline \indent
	Technical University of Denmark,\newline \indent
	Matematiktorvet 303B, 2800 Kgs. Lyngby, Denmark.}
\email{pabe@dtu.dk }
	%    author one information
	\author{Prasant Singh}
\address{Department of Applied Mathematics and Computer Science,\newline \indent
	Technical University of Denmark,\newline \indent
	Matematiktorvet 303B, 2800 Kgs. Lyngby, Denmark.}
%\thanks{The second named author was partially supported by a doctoral fellowship from the Council of Scientific and Industrial Research, India.}
%\curraddr{}
\email{psinghprasant@gmail.com}

	\date{\today}
	
\begin{abstract}
In this article we consider linear codes coming from skew-symmetric determinantal varieties, which are defined by the vanishing of minors of a certain fixed size in the space of skew-symmetric matrices. In odd characteristic, the minimum distances of these codes are determined and a recursive formula for the weight of a general codeword in these codes is given. %It is shown that these codes are a generalization of the Grassmann code associated to the line Grassmannian  and several initial higher weights of these codes are also given.

\end{abstract}

%determinantal varieties; skew-symmetric matrices; linear codes; minimum distance
%

\maketitle
\section{Introduction}
Let $\FF_q$ be the finite field with $q$ elements. From a mathematical point of view, an $\FF_q$-linear error-correcting code is a subspace of the vector space $\FF_q^n$. Algebraic varieties $V$ defined over $\FF_q$ are a rich source of such codes and various constructions of codes from a given variety exist. A very natural and much-studied construction of codes uses the points in $V(\FF_q)$, the set of $\FF_q$-rational points of $V$, as columns of a matrix. A code is then obtained by considering this matrix as generator matrix of the code. To construct the matrix, an order of the points will need to be chosen. In case the variety $V$ is contained in a projective space $\PP^{k-1}$, a choice of representatives of the $\FF_q$-rational points also needs to be made. Allowing any family of $\FF_q$-rational points of $\PP^{k-1}$ in this setup, leads to a more general construction of codes studied in \cite{TVN}. There, such a family of points was called a projective system. In the setting of projective systems, it was observed that a different choices of ordering and representatives of the points, give rise to equivalent codes. In particular the minimum distance and weight distribution of the resulting codes are independent on these choices. Another observation from \cite{TVN} is that if the projective system is not contained in a hyperplane of $\PP^{k-1}$, then the dimension of the resulting code is $k$.

Let $C \subset \FF_q^n$ be an $\FF_q$-linear code arising in this way from a projective variety $V \subset \PP^{k-1}$ defined over $\FF_q$. The Hamming weight $\w(c)$ of a codeword $c=(c_1,\dots,c_n) \in C$ is defined as $\w(c):=\#\{i \, : \, c_i \neq 0 \}.$ However, by construction $n-\w(c)$, then equals the number of common $\FF_q$-rational points on the intersection of $V$ and a certain hyperplane $H$ defined over $\FF_q$ depending on $c$. Therefore questions about the possible weights of codewords, can be rephrased in terms of intersections of $\FF_q$-rational hyperplanes with $V$. Codes coming from the Grassmannian variety have been studied from this point of view in \cite{Ryan,Ryan2,N,GL,BP}. Similarly, toric varieties were used to construct codes in \cite{Hansen,Diego}, flag varieties were used in \cite{Rodier}, and so on. For an overview see for example \cite{Little} and the references therein. Using classical determinantal varieties of generic matrices, a class of codes called determinantal codes were introduced and studied in \cite{BGH,BG}. As was shown there, determinantal codes have relatively few weights. More precisely, let $\ell \le m$ be natural numbers and $\mathbf{Det}(t,\ell, m)$ be the projective variety consisting of $\ell\times m$ matrices with coefficients in $\FF_q$ of rank $\leq t$. Then the corresponding code has at most $\ell$ weights. Apart from the determinantal varieties of generic matrices, other classically studied determinantal varieties are associated to for example symmetric and skew-symmetric matrices. For more details about these varieties one may refer to \cite{HT,JLP,S}. Inspired by this, we consider in this article codes associated to skew-symmetric determinantal varieties $\DetVar$ consisting of all $m \times m$ skew symmetric matrices with coefficients in $\FF_q$. We will call these codes \emph{skew determinantal codes}. The lengths and the dimensions of these codes are easily determined, since the number of $\FF_q$-rational points on $\DetVar$ is well known and the variety is nondegenerately embedded in $\PP^{\binom{m}{2}-1}$. However, the determination of the minimum distance requires some work. In fact, we will compute all possible nonzero weights a codeword of a skew determinantal code can have and then determine the least nonzero weight among them. Equivalently, we determine the possible number of $\FF_q$-rational intersection points that an $\FF_q$-rational hyperplane and $\DetVar$ can have. It turns out that like determinantal codes, only few possibilities can occur, namely at most $\lfloor m/2 \rfloor$. %After giving an iterative formula for the weights of the codewords, we determine the minimum distance of the code. Since in some cases, skew determinantal codes can be viewed as Grassmann codes, we also draw a conclusion on the initial higher weights of certain Grassmann codes.

%In this article, they showed that the determinantal codes can have only a few distinct weights. They also computed the minimum distance of the determinantal code associated to the determinantal variety of rank $1$ matrices. More precisely, let $\ell\leq m$ be two positive integers and $\mathbb{M}_{\ell\times m}$ be the set of all $\ell\times m$ matrices over $\Fq$. Let $\widehat{\mathbf{D}}_t(\ell, m)$ be the projective variety consisting of $\ell\times m$ matrices of rank $\leq t$ and $\widehat{C}_{det}(t; \ell, m)$ be the corresponding determinantal code.  Then it was proved that  $\widehat{C}_{det}(t; \ell, m)$ is an $[n, k, d]_q$ code, where
%\begin{equation}
%\label{eq: parofdetcode}
%n= \sum\limits_{j=1}^t\mu_j(\ell, m), \qquad k= \ell m\qquad \text{ and } d= q^{\ell+m-2}, \text{ when }t=1.
%\end{equation}
% and
% $$
% \mu_j(\ell, m)= \frac{q^{{j\choose 2}}}{q-1}\prod_{i=0}^{j-1}\frac{(q^{\ell-i}-1)(q^{m-i}-1)}{(q^{i+1}-1)}.
% $$
%They also showed that corresponding to every codeword of the determinantal code, one can associate a matrix and the weight of the codeword depends only on the rank of the associated matrix. In \cite{BG}, Beelen-Ghorpade proved that the minimum weight codewords of the determinantal code correspond to rank $1$ matrices. Further, they proved that the minimum distance of the determinantal code $\widehat{C}_{det}(t; \ell, m)$ is
%\begin{equation}
%\label{eq: mindisdetcode}
%d= q^{\ell+m-2}\nu_{t-1}(\ell-1, m-1), \text{ where }\nu_{t-1}(\ell-1, m-1)=1+(q-1)|\widehat{\mathbf{D}}_t(\ell, m)|.
%\end{equation}\\

\section {Preliminaries: Skew-symmetric Determinantal Varieties}

We begin this section by recalling the definition of skew-symmetric determinantal varieties.
Let $\Fq$ be a finite field with $q$ elements and let $m$ be a positive integer. By a \emph{skew-symmetric} or \emph{anti-symmetric} matrix $A$ of size $m$ over $\Fq$, we always mean that $A$ is an $m\times m$ matrix over $\Fq$ with all diagonal entries zero and $A^T= -A$. If the characteristic of the field $\Fq$ is not $2$ then the condition $A^T= -A$ is enough for $A$ to be skew-symmetric. Let $\Am$ be the set of all skew-symmetric matrices of size $m$ over $\Fq$. It is well known that $\Am$ is an $\FF_q$-vector space of dimension $\binom{m}{2}$ and hence one can think of $\Am$ as the affine space $\mathbb{A}(\Fq)^{\binom{m}{2}}$.

Let $t$ be an integer satisfying $t\leq m$. We define
\[
\DetAm := \{A\in\Am \,:\, \rank(A)\leq t\}
\]
and
\[
A(t, m) :=\{A\in\Am: \rank(A)= t\}.
\]
Note that $\DetAm=\emptyset$ for $t<0$ and $A(t, m) = \DetAm \setminus \mathbb{A}(t-1, m).$

It is well known that the rank of a skew-symmetric matrix is even (see for example \cite{Albert}). Therefore we have, $A(2t+1, m)=\emptyset$ for every nonnegative integer $t$. Consequently, we get $\mathbb{A}(2t, m)= \mathbb{A}(2t+1, m).$ By definition we have
\[
\DetAmt =\bigcup\limits_{r=0}^t\Amr
\]
and the union is disjoint.  If $\narm$ and $\Natm$ denote the cardinality of the sets $\Amr$ and $\DetAmt$, then
\begin{equation}
\label{eq: splitlength}
\Natm = \sum\limits_{r=0}^{t} \narm.
\end{equation}
Explicit expressions for $\narm$ in terms of $r,m$ and $q$ are well known in the literature (for odd characteristic see Theorem 3 in \cite{Carlitz}, for even characteristic see \cite{MacW} or Chapt.~15, \S 2, Theorem 2 in \cite{MacWSl}). Indeed, $n_a(0, m)= 0$ and for any $r\neq 0$
\begin{equation}
\label{eq:NumbSkewMat}
\narm = q^{r(r-1)}\frac{\prod\limits_{i=0}^{2r-1}(q^{m-i}-1)}{\prod\limits_{i=0}^{r-1}(q^{2(r-i)}-1)}.
\end{equation}

Let $\PP(\Am)$ be the projective space over $\Am.$ Since $\Am$ is an $\binom{m}{2}$ dimensional vector space, we can write $\PP(\Am)=\Pam.$ For any non-zero matrix $A\in\Am$ we denote by $[A]$ the corresponding homogeneous point of $\Pam$. More precisely, if $A=(a_{ij}) \in \Am \setminus \{0\},$ then $[A]=[a_{12}:\dots:a_{m-1 \, m}],$ with all $a_{ij}$ such that $i<j$ occurring as coordinates in $[A].$
Let $\DetVar$ be the  image of the set $\DetAmt$  under this natural map $\Am\setminus\{0\}\to \Pam$. The set $\DetVar\subseteq \PP(\Am)$ is a projective variety. More precisely, let $\X=(X_{ij})_{m\times m}$ be an $m\times m$ matrix in $m^2$ indeterminates $X_{ij}$ over $\Fq$. Then $\DetVar$ is given as the zero locus of all $2t+1$ minors of $\X$ and linear polynomials $X_{ij}+X_{ji}$ and $X_{ii}$ for every $1\leq i \le j\leq m$. Note that this describes $\DetVar$ as a subset of $\PP^{m^2-1},$ but the linear equations $X_{ij}+X_{ji}$ and $X_{ii}$ determine an $\binom{m}{2}-1$ dimensional projective subspace of $\PP^{m^2-1},$ which we previously had identified with $\PP(\Am).$ It is not hard to show that if $t>0$, then $\DetVar$ is not contained in a hyperplane of $\Pam$, or in other words that $\DetVar$ is nondegenerately embedded in $\Pam$. Indeed let a pair $(k,\ell)$ be given such that $1 \le k < \ell \le m$. Then the skew-symmetric matrix $A$ defined by $A_{k,\ell}=1,A_{\ell,k}=-1$ and $A_{ij}=0$ otherwise, is mapped to the projective point $[A]=[0:\dots:0:1:0:\dots:0] \in \Pam$ with a $1$ in the position corresponding to the pair $(k,\ell).$ Hence $\DetVar$ contains $\binom{m}{2}$ projective points in general position, implying that it is not contained in a hyperplane.

We call $\DetVar$ a \emph{skew-symmetric determinantal variety}. One can indeed show that it is a variety in the usual sense, but we will not need this fact here. Let $\Lat$ denote the number of $\FF_q$-rational points of $\DetVar$, then $\Lat$ is given by
\begin{equation*}
%\label{eq:length}
\Lat =\dfrac{\Natm-1}{q-1},
\end{equation*}
where the value of $\Natm$ is determined by equations \eqref{eq: splitlength} and \eqref{eq:NumbSkewMat}.

\section{Linear Codes Associated to the Determinantal Variety $\DetVar$}

Using the $\FF_q$-rational points of $\DetVar$ as a projective system, yields a linear code $\CAtm$ which we call a skew determinantal code. In order to make this precise, let $N:=\Lat$ and let $B_1, \ldots, B_N$ be representatives of all the $\FF_q$-rational points of $\DetVar$. Note that for all $i$ we have $B_i \in \FF_q^{\binom{m}{2}}.$ An element $B =(b_{ij})_{1 \le i < j \le m} \in \FF_q^{\binom{m}{2}}$ gives rise to a unique skew-symmetric matrix $A=(a_{ij})$ by setting $a_{ij}:=b_{ij}$ if $i<j$, $a_{ii}:=0$, and $a_{ij}:=-b_{ji}$ if $i>j$. Therefore we will with slight abuse of notation identify $\FF_q^{\binom{m}{2}}$ with the space of skew-symmetric $m \times m$ matrices.

By construction, the matrix $G(2t,m)$ with columns $B_1,\dots,B_N$ is a generator matrix of $\CAtm$. Clearly therefore, the length of $\CAtm$ equals $N=\Lat$. Further, since from the previous section, we know that the $\FF_q$-rational points of $\DetVar$ are not contained in any hyperplane of $\Pam$, the dimension of $\CAtm$ equals $\binom{m}{2}.$ In this section, we will determine the minimum distance of $\CAtm$ in case $q$ is odd. The case $q$ is even seems more involved and could be interesting future work. For the remainder of this article $m$ and $t$ will be assumed to be integers such that $m>0$ and $0 \le 2t \le m$.

Another way to describe $\CAtm$ is as the image of an evaluation map. Let $\Fq[\X]_1$ denotes the vector space of linear homogeneous polynomials in variables $X_{ij}$, $1\leq i < j\leq m.$ Consider the evaluation map
\begin{equation*}
%\label{eq: def}
\Ev: \Fq[\X]_1\to \mathbb{F}_q^N\text{ defined by }f(\X)=\sum f_{ij} X_{ij}\mapsto (f(B_1),\ldots, f(B_N)).
\end{equation*}
The evaluation map $\Ev$ defined above is a linear map. Moreover, varying the pairs $(i,j)$ with $1 \le i < j \le m$, $\Ev(X_{ij})$ are precisely the $\binom{m}{2}$ rows of the matrix $G(2t,m).$ Therefore the image of $\Ev$ equals $\CAtm.$ Since $\FF_q[\X]_1$ is an $\binom{m}{2}$-dimensional vector space over $\FF_q$, the map $\Ev$ is a bijection. In order to determine the minimum distance of $\CAtm$ we consider a related code $\C$ defined as the image of another evaluation map
$$
\widehat{\Ev}:\Fq[\X]_1\to \mathbb{F}_q^{\Natm}\text{ defined by }f(\X)\to (f(A))_{A\in\DetAmt}.
$$
Note that in the definition of $\widehat{\Ev}$ we implicitely used the identification of elements in $\FF_q^{\binom{m}{2}}$ and $m \times m$ skew-symmetric matrices mentioned in the beginning of this section. The weights of the codewords in $\CAtm$ and $\C$ are directly related to each other as the following easy lemma shows.
\begin{lemma}\label{lemma: aff-proj}
Let $f\in\Fq[\X]_1$ be a polynomial. Then the Hamming weights of the codewords $\Ev(f)$ and $\widehat{Ev}(f)$ satisfy
$$
\w(\widehat{\Ev}(f)) = (q-1) \w(\Ev(f)).
$$
\end{lemma}
\begin{proof}
The proof of the lemma is a simple consequence of the fact that for every $f\in\Fq[\X]_1$ and any $A \in \DetAmt$, we have
	\[
	f(A)\neq 0\iff f(\alpha A)\neq 0 \text{ for all } \alpha \in \mathbb{F}_q^{*}.
	\]
	
	\end{proof}
In view of the above lemma, in order to compute the minimum distance of the code $\CAtm$, it is enough to calculate the minimum distance of the code $\C$.

To proceed further, let us write $M:= N_a(2t, m)$ and $\DetAmt=\{A_1,\ldots, A_M\}$ in some fixed order. Note that codewords of $\C$ are indexed by the set $\DetAmt$ therefore, for every codeword $c\in\C$ and  any $A_i\in\DetAmt$ we use the notation $c(A_i)$ to denote the $A_i^{\it th}$ coordinate of the codeword $c$. From now on we always fix $\Fq$ as a finite field with characteristic of $\Fq$ not equal to $2$.

\begin{theorem}\label{thm:skewsymm}
For any codeword $(c_A)_{A \in \DetAmt}\in\C$, there exist a unique skew-symmetric matrix $F\in\Am$ such that
\[
c_A= -\tr(FA) \text{ for every }A\in \DetAmt.
\]
\end{theorem}

\begin{proof}
Let $c\in\C$ and let $f\in \FF_q[\X]_1$ be the linear polynomial such that $c=\widehat{\Ev}(f).$ In particular $c_A=f(A).$ Writing
$f=\sum_{i<j} f'_{ij}X_{ij}$, define the $m \times m$ matrix $F:=(f_{ij})$ by $f_{ij}:=f'_{ij}/2$ if $i<j$, $f_{ii}:=0$ and $f_{ij}:=-f'_{ji}/2$ if $i>j$. Note that $F$ is skew-symmetric. We claim that for any $A=(a_{ij}) \in \DetAmt$ it holds that $f(A)=-\tr(FA),$ which would show the existence of the matrix $F$ in the theorem. Indeed, we have
\begin{multline*}
f(A)= \sum_{i<j}f'_{ij}a_{ij}=\sum_{i<j}(f_{ij}-f_{ji})a_{ij}=\sum_{i<j}f_{ij}a_{ij}-\sum_{i<j}f_{ji}a_{ij} \\
=-\sum_{i<j}f_{ij}a_{ji}-\sum_{j>i}f_{ji}a_{ij}=-\sum_{i=1}^m \sum_{j=1}^m f_{ij}a_{ji}=-\tr(FA).
\end{multline*}
Here we used $a_{ij}=-a_{ji}$ in the fourth equality and $a_{ii}=0$ in the fifth.

To show uniqueness, suppose that there exist two skew-symmetric matrices $F=(f_{ij})$ and $G=(g_{ij})$ such that $\tr(FA)=\tr(GA)$ for all $A \in \DetAmt.$ For $k < \ell,$ define the matrix $E(k,\ell)$ as the matrix with zero entries everywhere except for the coordinates $(k,\ell)$, respectively $(\ell,k)$, where the matrix has entry $1$, respectively $-1$. Since $t \ge 1$, we have $E(k,\ell) \in \DetAmt.$ Moreover, $\tr(FE(k,\ell))=-f_{k\ell}+f_{\ell k}=-2f_{k\ell}$ and similarly $\tr(FE(k,\ell))=-2g_{k \ell}.$ Hence $F=G$ follows.
\end{proof}

Note that the assumption that $q$ is odd, was crucial in the proof of the theorem. Also note that since both $\FF_q[\X]_1$ and the space of skew-symmetric matrices $\Am$, are vector spaces of the same dimension $\binom{m}{2},$ any word of the form $(-\tr(FA))_A$ is a codeword of $\C.$ As a consequence of the theorem, we will see that the code $\C$ (and hence $\CAtm$) only can have few weights.
\begin{corollary}\label{cor:feweights}
If $c=(-\tr(FA))_{A \in \DetAmt} \in \C$ for a skew-symmetric matrix $F$, then the Hamming weight of $c$ only depends on the rank of $F$. Moreover, there are at most $\lfloor m/2 \rfloor$ possibilities for the Hamming weight of a nonzero codeword $c \in \CAtm.$
\end{corollary}
\begin{proof}
Let $c =(c_A)_{A \in \DetAmt} \in \C$ be a nonzero codeword. By Theorem \ref{thm:skewsymm}, there exists a nonzero skew-symmetric matrix $F$ such that $c_A=-tr(FA)$ for all $A \in \DetAmt.$ By Theorem 4 in \cite{Albert} or alternatively Chapter XV, Corollary 8.2 in \cite{Lang}, there exist a nonsingular $m\times m$ matrix $L$ such that
\[
L F L^T= \begin{bmatrix}
E      &   \\
& E &  \\
& & \ddots &  \\
& &  & E& \\
\hdotsfor{5} \\
0       & 0 & 0 & \dots &0
\end{bmatrix}, \text{ with } E= \begin{bmatrix}
0       & 1 \\
-1       & 0
\end{bmatrix}.
\]
The number of occurrences of the matrix $E$ depends on the rank of $F$. Specifically, if the rank of $F$, which necessarily is even, equals $2k$ for some $1 \le k \le m/2$, then the matrix $E$ occurs $k$ times.
Since $L$ is nonsingular, the mapping $A\mapsto L^TAL$ is a permutation of $\DetAmt$. For every matrix $A$, we have $\tr((L^TFL)A)= \tr(F(L^TAL))$. Therefore, the Hamming weights of the codewords $c$ and $(\tr((L^TFL)A))_{A \in \DetAmt}$ are the same. In particular, the Hamming weight of $c$ only depends on the rank of $F$. The first part of the corollary now follows.

Since $k$ needs to be an integer satisfying $1 \le k \le m/2,$ we see that there are only $\lfloor m/2 \rfloor$ possible ranks for $F$. This shows that there are at most $\lfloor m/2 \rfloor$ possibilities for the Hamming weight of $c$. By Lemma \ref{lemma: aff-proj}, the second part of the corollary follows.
\end{proof}

To determine the minimum distance of $\CAtm,$ we need to determine which of the possible $\lfloor m/2 \rfloor$ weights is the smallest. In order to do that, let us fix some notations. If $c=(-\tr(FA))_{A \in \DetAmt} \in \C$ for a skew-symmetric matrix $F$ of rank $2k$, we define $\Wk:=\w(c)$. By Corollary \ref{cor:feweights}, this is a valid definition, since $\w(c)$ does not depend on the choice of $F$.  For $0 \le 2k \le m,$ define the skew-symmetric matrix of rank $2k$
\[
E_{2k}:=\begin{bmatrix}
E      &   \\
& E &  \\
& & \ddots &  \\
& &  & E& \\
\hdotsfor{5} \\
0       & 0 & 0 & \dots &0
\end{bmatrix},
\]
with $E$ as in the proof of Corollary \ref{cor:feweights} occurring exactly $k$ times.
Then
$$\Wk=|\{A \in \DetAmt \,:\, \tr(E_{2k}A) \neq 0\}|.$$
If we define
\[\supp_{2k}(2r,m):=\{A\in A(2r, m) \,: \, \tr(E_{2k}A)\neq 0\} \text{ and }
\Wkr:= \lvert \supp_{2k}(2r,m) \rvert,
\]
then
\begin{equation}\label{eq: splitwt}
\Wk=\sum\limits_{r=1}^t\Wkr.
\end{equation}
Our next goal is to find recursive formulas for $\Wk$, that we will use as the key ingredient to show which of the $\Wk$ is the smallest. Note that this approach is inspired by \cite{BG}, where a similar approach was developed to determine the minimum distance of the determinantal code introduced in \cite{BGH}.

\begin{theorem}
	\label{Thm: Iterativeweight}
	Let $m$ and $1\leq 2k\leq m$  be fixed. Let $2r\leq m$ and $\Wkr$ be as defined above.	Then
	\begin{equation*}
%	\label{eq: restwt}
	\begin{split}
	\Wkr & = q^{2r}\mathrm{w}_{2k-2}(2r, m-2) + (q-1)q^{2r-1}\left(n_a(2r, m-1) - n_a(2r, m-2)\right)\\
	&  + (q-1)q^{m-2}n_a(2r-2, m-1)- q^{2r-2}\mathrm{w}_{2k-2}(2r-2, m-2)  \\
	& - (q-1)q^{2r-3}\left(n_a(2r-2, m-1) - n_a(2r-2, m-2)\right).
	\end{split}
	\end{equation*}
\end{theorem}

\begin{proof}
Given a matrix $A \in A(2r,m)$, denote by $\A$ the matrix obtained by deleting the $(2k)^{\it th}$ row and column of the matrix $A$. Then $\A$ is a skew-symmetric matrix of rank between $2r-2$ and $2r$. Since the rank of a skew-symmetric matrix is even, the rank of $\A$ can be either $2r$ or $2r-2$. Now consider the map
	\[
	\phi: \supp_{2k}(2r, m) \to A(2r, m-1)\bigcup A(2r-2, m-1),\text{ defined by }A\mapsto \A.
	\]
We use this map to count the cardinality of the set $\supp_{2k}(2r, m)$ which is $\Wkr$. Using the map $\phi$ we get,
	\[
	\Wkr=\sum\limits_{\A\in A(2r, m-1)}|\phi^{-1}(\A)| \;+\; \sum\limits_{\A\in A(2r-2, m-1)}|\phi^{-1}(\A)|.
	\]
	
	For any matrix $\A$, let $\Ak$ denote the $(2k-1)^{\it th}$ row of $\A$. We divide each sum in the above expression in two different terms, depending on whether $\Ak$ is zero or non-zero. Rewriting the above expression, we get
	
	\begin{equation}
	\label{eq:fiberexp}
	\begin{split}
	\Wkr & =\sum _{\substack{\A\in A(2r, m-1) \\ \Ak=0}}|\phi^{-1}(\A)| \;+\; \sum _{\substack{\A\in A(2r, m-1)  \\ \Ak\neq 0}}|\phi^{-1}(\A)|\\
	& +  \sum _{\substack{\A\in A(2r-2, m-1)  \\ \Ak=0}}|\phi^{-1}(\A)|\;+ \; \sum _{\substack{\A\in A(2r-2, m-1) \\ \Ak\neq 0}}|\phi^{-1}(\A)|.
	\end{split}
	\end{equation}
	
We divide the 	counting of the fibers $\phi^{-1}(\A)$ in four different cases corresponding to the four summations in this equation. Given a row or column vector $v=(v_1,\dots,v_m) \in \FF_q^m$, we will use the phrase \emph{shortened row or column vector} for the vector $v' \in \FF_q^{m-1}$ obtained from $v$ by deleting its $(2k)^{\it th}$ coordinate. In this way, we can for example say that the rows of $\A$ are obtained by shortening rows of the matrix $A$ and likewise for columns.

\begin{case}
Let $\A\in A(2r, m-1)$ and $\Ak=0$.  Let $A\in \phi^{-1}(\A)$. Since $A$ and $\A$ have the same rank $2r$, we conclude that the shortened $(2k)^{\it th}$ column of $A$ lies in the column space of $\A$. This in particular means that the entry $A_{2k-1\,2k}$ is a linear combination of the entries in $\Ak$ and hence zero. Hence, $\tr(E_{2k-2}\A)=\tr(E_{2k}A)\neq 0$, where the inequality follows by our assumption that $A \in \supp_{2k}(2r,m)$. Therefore, the first summation in equation \eqref{eq:fiberexp} runs over $\mathrm{w}_{2k-2}(2r, m-2)$ many matrices $\A$ and for any such $\A$ we get $|\phi^{-1}(\A)|=q^{2r}$, since the dimension of the column space of $\A$ equals $2r$. Therefore in total, we get
\begin{equation}\label{eq:case1}
		\sum _{\substack{\A\in A(2r, m-1) \\ \Ak=0}}|\phi^{-1}(\A)| = q^{2r}\mathrm{w}_{2k-2}(2r, m-2).
\end{equation}
\end{case}

\begin{case}
Let $\A\in A(2r, m-1)$ and $\Ak\neq 0$. As in the previous case, for any $A\in \phi^{-1}(\A)$ the shortened $(2k)^{\it th}$ column of $A$ must lie in the column span of $\A$.
Moreover, we require that $2A_{2k-1\, 2k}\neq-\tr(E_{2k-2}\A)$.
Since $\Ak$ is assumed to be non zero, the projection map from the column space of $\A$ to the $(2k-1)^{\it th}$ coordinate is nonzero. Since the column space of $\A$ has dimension $2r$, we see that for a given $\A$,  there are exactly $q^{2r}-q^{2r-1}$ possibilities to choose the $(2k)^{\it th}$ column of $A$ such that $2A_{2k-1\, 2k}\neq-\tr(E_{2k-2}\A).$ In other words: $|\phi^{-1}(\A)|=q^{2r}-q^{2r-1}$. Further, the number of matrices $\A\in A(2r, m-1)$ such that $\Ak\neq 0$ is given by $n_a(2r, m-1)- n_a(2r, m-2)$. In total we get
		\begin{equation}
		\label{eq:case2}
		\sum _{\substack{\A\in A(2r, m-1) \\ \Ak\neq 0}}|\phi^{-1}(\A)|=	(q-1)q^{2r-1}(n_a(2r, m-1)- n_a(2r, m-2)).
		\end{equation}

	\end{case}

	\begin{case}
		Let $\A\in A(2r-2, m-1)$ and $\Ak=0$. This is the most complex case therefore, we divide the counting in several subcases. \newline
		First, we count the number of  $A\in \phi^{-1}(\A)$ satisfying $A_{2k-1\,2k}=0$. Since $A_{2k-1\,2k}=0$, we have $m-2$ positions in the $(2k)^{\it th}$ column of $A$ that are undetermined. Moreover, the shortened $(2k)^{\it th}$ column of $A$ can not be in the column span of $\A$.
Since $\A$ is of rank $2r-2$ and $\A_{2k-1}=0$, this leaves exactly $q^{m-2}-q^{2r-2}$ possibilities for the matrix $A$. %many such matrices $A$ in the fiber $\phi^{-1}(\A)$ satisfying $A_{2k-1\,2k}=0$.
Further, in all these cases we get $\tr(E_{2k-2}\A)=\tr(E_{2k}A)\neq 0$.
Therefore, there are $\mathrm{w}_{2k-2}(2r-2, m-2)$ many possibilities for $\A$.

Second, we count the number of  $A\in \phi^{-1}(\A)$ satisfying $A_{2k-1\,2k}\neq 0$. We further divide this in two parts depending on $\tr(E_{2k-2}\A)$ being zero or non-zero. If $\tr(E_{2k-2}\A)=0$, then $\tr(E_{2k}A)=2A_{2k-1\,2k} \neq 0.$ Hence, the only restriction we have is that the shortened $(2k)^{\it th}$ column of $A$ can not be in the column span of $\A$. However, since we assign a nonzero value to $A_{2k-1\,2k}$, this is guaranteed. The remaining positions can be chosen arbitrarily. Therefore, in this case we get $n_a(2r-2, m-2)-\mathrm{w}_{2k-2}(2r-2, m-2)$ many matrices $\A$ with  $\tr(E_{2k-2}\A)=0$ and for any such $\A$, the cardinality of the fiber is $(q-1)q^{m-2}$.

If $\tr(E_{2k-2}\A)\neq 0$, then in addition to $A_{2k-1\,2k}\neq 0$ we require $2A_{2k-1\,2k}\neq -\tr(E_{2k-2}\A)$, leaving $q-2$ possible values for $A_{2k-1\,2k}$. We have $\mathrm{w}_{2k-2}(2r-2, m-2)$ many possibilities for the matrix $\A$, since we assumed $\tr(E_{2k-2}\A)\neq 0$. For a given $\A$, we have $(q-2)q^{m-2}$ many matrices $A$ in the fiber $\phi^{-1}(\A)$. Adding all together, we obtain
		\begin{equation}
		\label{eq:case3}
		\begin{split}
		\sum _{\substack{\A\in A(2r-2, m-1) \\ \Ak=0}}|\phi^{-1}(\A)| & =(q^{m-2}-q^{2r-2})\mathrm{w}_{2k-2}(2r-2, m-2) \\
		&+  (q-1)q^{m-2}\left(n_a(2r-2, m-2)-\mathrm{w}_{2k-2}(2r-2, m-2)\right)\\
		&+ (q-2)q^{m-2}\mathrm{w}_{2k-2}(2r-2, m-2).
		\end{split}
		\end{equation}
		
\end{case}
	
\begin{case}
Finally, let $\A\in A(2r-2, m-1)$ and $\Ak\neq0$. Since $\A$ is of rank $2r-2$, the shortened $(2k)^{\it th}$ column of $A$ must not lie in the column span of $\A$. Further, $2A_{2k-1\,2k}\neq-\tr(E_{2k-2}\A)$. Like in Case $2$, we consider the projection map on the $(2k-1)^{\it th}$ coordinate. First of all, we require $2A_{2k-1\,2k}\neq-\tr(E_{2k-2}\A)$, leaving $q^{m-1}-q^{m-2}$ a priori possibilities for the $(2k)^{\it th}$ column of $A$. However, since the shortened $(2k)^{\it th}$ column of $A$ cannot lie in the column span of $\A$, $q^{2r-2}-q^{2r-3}$ many of these possibilities need to be excluded. This shows that for a given $\A$ as above, $|\phi^{-1}(\A)|=(q^{m-1}-q^{m-2})-(q^{2r-2}-q^{2r-3}).$
Also, there are $n_a(2r-2, m-1)- n_a(2r-2, m-2)$ many matrices $\A$ satisfying $\A\in A(2r-2, m-1)$ and $\Ak\neq0$. All together, we get
\begin{equation}\label{eq:case4}
		\sum _{\substack{\A\in A(2r-2, m-1) \\ \Ak\neq0}}|\phi^{-1}(\A)|=(q-1)(q^{m-2}-q^{2r-3})(n_a(2r-2, m-1)- n_a(2r-2, m-2)).
\end{equation}
\end{case}
Now the theorem follows from equations \eqref{eq:fiberexp},\;  \eqref{eq:case1},\;\eqref{eq:case2},\;\eqref{eq:case3}, and \eqref{eq:case4}.
\end{proof}
We will use Theorem \ref{Thm: Iterativeweight} to find expressions for the weights $\Wk$ of the codewords in $\C.$
It will be convenient for every $0\leq r\leq t$ to introduce the quantity
\begin{equation}\label{eq:P}
P_m(2k, 2r):= q^{2r}\mathrm{w}_{2k-2}(2r, m-2) + (q-1)q^{2r-1}\left(n_a(2r, m-1) - n_a(2r, m-2)\right).
\end{equation}
Note that $P_m(2k, 0)=0$. Theorem \ref{Thm: Iterativeweight} has the following corollary.
\begin{corollary}\label{cor: wtofcodeword}
Let $t$ and $k$ be an integers such that $0 < 2t \le m$ and $0 < 2k \le m.$  Then
	\begin{equation*}
	%\label{eq: wtofcodeword}
	\Wk= P_m(2k, 2t) + (q-1)q^{m-2}N_a(2t-2, m-1).
	\end{equation*}
\end{corollary}
\begin{proof}
Using Theorem \ref{Thm: Iterativeweight} and the quantity $P_m(2k, 2r)$ defined in equation \eqref{eq:P}, a direct computation shows that
\begin{equation*}
%\label{eq:affinewt}
\Wkr= P_m(2k, 2r)-P_m(2k, 2r-2) + (q-1)q^{m-2}n_a(2r-2, m-1).
\end{equation*}
Then equation \eqref{eq: splitwt} implies that
\[
\Wk=P_m(2k,2t)-P_m(2k,0)+(q-1)q^{m-2}\sum_{r=1}^t n_a(2r-2, m-1).
\]
The corollary now follows.	
\end{proof}

\begin{theorem}
	\label{Thm: TheMinWt}
	The minimum distance of the code $\CAtm$ is given by
$$\dfrac{\mathrm{W}_2(2t, m)}{q-1}=(q^{m-2t}-1)q^{m+2t-4} n_a(2t-2, m-2)\; +\; q^{m-2}N_a(2t-2, m-1).$$
\end{theorem}
\begin{proof}
First we show that $\mathrm{W}_2(2t, m)$ is the minimum distance of $\C$. Lemma \ref{lemma: aff-proj} then implies that $\mathrm{W}_2(2t, m)/(q-1)$ is the minimum distance of $\CAtm.$

We know that the weight of a non-zero codeword of $\C$ is among $\Wk$ where $1\leq k\leq \floor*{\frac{m}{2}}$. %Let $1\leq 2k\leq m$ be be such a integer.
From Corollary \ref{cor: wtofcodeword}, we get
	\[
	\Wk -\mathrm{W}_2(2t, m) = P_m(2k, 2t) - P_m(2, 2t).		
	\]
Using equation \eqref{eq:P} and taking into account that $\mathrm{w}_0(2t, m-2)=0$, we get
\begin{equation}\label{eq:W2k_W2}
	\Wk -\mathrm{W}_2(2t, m)= q^{2t}\mathrm{w}_{2k-2}(2t, m-2).
\end{equation}
	In particular, we get $\Wk\geq \mathrm{W}_2(2t, m)$ for every $1\leq k\leq \floor*{\frac{m}{2}}$ and hence $\mathrm{W}_2(2t, m)$ is the minimum distance of the code $\C$.

To finish the proof, it is sufficient to show that
	\[
	\mathrm{W}_2(2t, m) =(q-1) (q^{m-2t}-1)q^{m+2t-4} n_a(2t-2, m-2)\; +\;(q-1) q^{m-2}N_a(2t-2, m-1).
	\]
Corollary \ref{cor: wtofcodeword} implies that
	\[
	\mathrm{W}_2(2t, m)\; =\; P_m(2, 2t) + (q-1)q^{m-2}N_a(2t-2, m-1).
	\]
Therefore, we only need to show that
	$$
	P_m(2, 2t) =(q-1)(q^{m-2t}-1)q^{m+2t-4} n_a(2t-2, m-2).
	$$
From equation \eqref{eq:P}, we have
	$$
	P_m(2, 2t)= q^{2t}\mathrm{w}_0(2t, m-2) +(q-1) q^{2t-1}\left(n_a(2t, m-1) - n_a(2t, m-2)\right).	
	$$
Since $\mathrm{w}_0(2t, m-2)=0$, equation \eqref{eq:NumbSkewMat} implies
	\begin{align*}
	P_m(2, 2t) &= (q-1)q^{2t-1}\left\{q^{t(t-1)}\dfrac{\prod\limits_{i=0}^{2t-1}(q^{m-1-i}-1)}{\prod\limits_{i=0}^{t-1}(q^{2(t-i)}-1)}- q^{t(t-1)}\dfrac{\prod\limits_{i=0}^{2t-1}(q^{m-2-i}-1)}{\prod\limits_{i=0}^{t-1}(q^{2(t-i)}-1)}\right\}\\
	&=(q-1)q^{2t-1}\left\{q^{t(t-1)}\dfrac{\prod\limits_{i=0}^{2t-2}(q^{m-2-i}-1)}{\prod\limits_{i=0}^{t-1}(q^{2(t-i)}-1)}    \left( q^{m-1}-q^{m-2t-1}\right)\right\}\\
	&=(q-1)q^{2t-1}\left\{q^{m-2t-1}\left(q^{m-2t}-1\right)q^{t(t-1)}\dfrac{\prod\limits_{i=0}^{2(t-1)-1}(q^{m-2-i}-1)}{\prod\limits_{i=0}^{(t-1)-1}(q^{2(t-i)}-1)}   \right\}\\
	&= (q-1)(q^{m-2t}-1)q^{m+2t-4} q^{(t-1)(t-2)}\dfrac{\prod\limits_{i=0}^{2(t-1)-1}(q^{m-2-i}-1)}{\prod\limits_{i=0}^{(t-1)-1}(q^{2(t-i)}-1)} \\
	&= (q-1)(q^{m-2t}-1)q^{m+2t-4} n_a(2t-2, m-2).
	\end{align*}
	This completes the proof of the theorem.
\end{proof}

\begin{remark}
Recall that $1 \le t \le \lfloor m/2 \rfloor.$
In the extremal choices of $t$, the code $\CAtm$ is equal to well known previously studied codes. For $t=1$, the determinantal variety $\DetVar$ is simply the line Grassmann variety $G(2,m).$ Hence in this case the code $\CAtm$ is a particular instance of the Grassmann codes studied in \cite{N}. In fact from \cite{N}, the complete weight enumerator of the code $C_{\mathbf{A}}(2, m)$ can be obtained. From Nogin's results it is easy to show that $\mathrm{W}_2(2, m)<\cdots < \mathrm{W}_{2 \lfloor m/2 \rfloor}(2, m).$

For $t=\lfloor \frac m2 \rfloor,$ we have $\DetVar=\Pam.$ Hence in this case $\CAtm$ is a first order projective Reed--Muller code. Projective Reed--Muller codes were introduced in \cite{Lachaud}. It is well known that first order projective Reed--Muller codes are constant weight codes. Indeed equation \eqref{eq:W2k_W2} implies that $\mathrm{W}_2(2 \lfloor m/2 \rfloor, m)=\mathrm{W}_{2 k}(2\lfloor m/2 \rfloor, m)$ for every $k$.
\end{remark}

It is in general not clear how the elements in the sequence $\Wk$, $2 \le 2k \le m$ are ordered. However, in case $1 < t < \lfloor m/2 \rfloor,$ we are able to determine the minimum among them.

\begin{theorem}
Suppose that $4 \le 2k \le m$ and $4 \le 2t \le m-2.$ Then $\mathrm{W}_2(2t, m) < \Wk.$ Moreover, the number of minimum weight codeword in $\CAtm$ equals $(q^m-1)(q^{m-1}-1)/(q^2-1)$. %$n_a(2,m).$
\end{theorem}
\begin{proof}
Using equation \eqref{eq:W2k_W2}, the first part of the theorem follows once we show that $\mathrm{w}_{2k-2}(2t, m-2)>0.$ Equation \eqref{eq:case4} in the proof of Theorem \ref{Thm: Iterativeweight} implies that
$$\mathrm{w}_{2k-2}(2t, m-2) \ge (q-1)(q^{m-4}-q^{2t-3})(n_a(2t-2, m-3)- n_a(2t-2, m-4)).$$
Using a similar computation as in the proof of Theorem \ref{Thm: TheMinWt}, we can rewrite the right-hand side of this inequality and obtain that
$$\mathrm{w}_{2k-2}(2t, m-2) \ge (q-1)(q^{m-2t-1}-1)(q^{m-2t}-1)q^{m+2t-8}n_a(2t-4, m-4).$$
Since $t \ge 2$, equation \eqref{eq:NumbSkewMat} implies that $\mathrm{w}_{2k-2}(2t, m-2)>0$.

The statement about the number of minimum weight codewords in $\CAtm$ is a direct consequence of the above and Theorem \ref{thm:skewsymm}, since $n_a(2,m)=(q^m-1)(q^{m-1}-1)/(q^2-1)$ by equation \eqref{eq:NumbSkewMat}.
\end{proof}

\section{Acknowledgements}

Prasant Singh is supported by the HC\O rsted-COFUND postdoctoral grant \emph{Understanding Schubert Codes}.

Peter Beelen would like to acknowledge the support from The Danish Council for Independent Research (DFF-FNU) for the project \emph{Correcting on a Curve}, Grant No.~8021-00030B.


\begin{thebibliography}{AAAA}

\bibitem{Albert} A.A.~Albert, {\rm Symmetric and alternate matrices in an arbitrary field I}, \emph{Trans. Amer. Math. Soc.} {\bfseries 43}  (1938),  no. 3, 386--436.

\bibitem{BGH}	P.~Beelen, S.R.~Ghorpade and S.U.~Hasan,  {\rm Linear codes associated to determinantal varieties}, \emph{Discrete Math}, {\bfseries 338} (2015), 1493--1500.
	
\bibitem{BG} P.~Beelen and S.R.~Ghorpade, {\rm Hyperplane Sections of Determinantal Varieties over Finite Fields and Linear Codes}, preprint available arXiv:1809.04690.
    	
\bibitem{BP} P.~Beelen and F.~Pi\~nero, {\rm The structure of dual Grassmann codes}, \emph{Des. Codes Cryptogr.} {\bfseries 79} (2016), 451--470.

\bibitem{Carlitz} L.~Carlitz, {\rm Representations by skew forms in a finite field}, \emph{Arch. Math.} (Basel) \textbf{5},  (1954), 19--31.

%\bibitem{GPP}	S.~R.~Ghorpade, A.~R.~Patil and H.~K.~Pillai, {\rm Decomposable subspaces,  linear sections of Grassmann varieties,	and higher weights of Grassmann codes}, 	{\em Finite Fields Appl.} {\bfseries 15} (2009),  54--68.	

\bibitem{GL}	S.R.~Ghorpade and G.~Lachaud,  {\rm Higher weights of Grassmann codes}, \emph{Coding Theory, Cryptography and Related Areas} (Guanajuato, 1998), J. Buchmann, T. H{\o}holdt, H. Stichtenoth and H. Tapia-Recillas Eds., Springer-Verlag, Berlin, 	2000, 122--131.

\bibitem{Hansen} J.~Hansen, {\rm Toric surfaces and error correcting codes}, \emph{Coding Theory, Cryptography and Related Areas} (Guanajuato, 1998), J. Buchmann, T. H{\o}holdt, H. Stichtenoth and H. Tapia-Recillas Eds., Springer-Verlag, Berlin, 	2000, 132--142.

\bibitem{HT}	J.~Harris and L.W.~Tu,  {\rm On symmetric and skew-symmetric determinantal varieties}, \emph{Topology} {\bfseries 23}, (1984), 71--84.
	
\bibitem{JLP}    T.~J\^{o}zefiak, A.~Lascoux and P.~Pragacz, {\rm Classes of determinantal varieties associated with symmetric and skew-symmetric matrices}, \emph{Izv. Akad. Nauk SSSR Ser. Mat.} {\bfseries 18}, (1981), 662--673.

\bibitem{Lachaud}
G.~Lachaud,  {\rm Projective  Reed--Muller  codes, Coding  theory  and  applications}, \emph{Lecture Notes in Comput. Sci.}, {\bfseries 311}, (1988), Springer, Berlin, 125--129.

\bibitem{Lang} S.~Lang, {\rm Algebra}, Third edition, Addison-Wesley Publishing Company, Reading, MA, 1994.

\bibitem{Little} J.B.~Little, {\rm Codes from Higher Dimensional Varieties}, Chapter 7 in \emph{Advances in Algebraic Geometry Codes. Series on Coding Theory and Cryptology} vol.5, World Scientific Publishing Co. Pte. Ltd., 257--293, 2008.

\bibitem{MacW} F.J.~MacWilliams, {\rm Orthogonal matrices over finite fields}, \emph{ Amer. Math. Monthly} {\bfseries 76} (1969), 152--164.

\bibitem{MacWSl} F.J.~MacWilliams and N.J.A.~Sloane, {\rm The Theory of Error-Correcting Codes}, North-Holland Mathematical Library, Vol. 16. North-Holland Publishing Co., Amsterdam-New York-Oxford, 1977.

\bibitem{N}	D.Yu.~Nogin, {\rm Codes associated to Grassmannians}, \emph{Arithmetic, Geometry and Coding Theory} (Luminy, 1993), R. Pellikaan, M. Perret, S. G. Vl\u{a}du\c{t}, Eds., Walter de Gruyter, Berlin,	(1996), 145--154.

\bibitem{Rodier} F.~Rodier,	{\rm Codes from flag varieties over a finite field}, \emph{J. Pure Appl. Algebra} {\bfseries 178} (2003), 203--214.

\bibitem{Diego} D.~Ruano, {\rm On the parameters of $r$-dimensional toric codes}, \emph{Finite Fields and Their Applications} {\bfseries 13} (2007), 962--976.

\bibitem{Ryan} C.T.~Ryan, {\rm An application of Grassmannian varieties to coding theory}, \emph{Congr. Numer.} {\bfseries 157} (1987), 257--271.

\bibitem{Ryan2} C.T.~Ryan, {\rm Projective codes based on Grassmann varieties}, \emph{Congr. Numer.} {\bfseries 157} (1987), 273--279.

\bibitem{S}	I.R.~Shafarevich, {\rm Basic Algebraic Geometry 1: Varieties in Projective Space}, Translated by M. Reid, Springer-Verlag Berlin Heidelberg, 2013.

\bibitem{TVN}	M.~Tsfasman, S.~Vl{\u{a}}du{\c{t}} and D.~Nogin, {\rm Algebraic Geometric Codes: Basic Notions}, Math. Surv. Monogr., vol. 139, Amer. Math. Soc., Providence, 2007.
\end{thebibliography}
\end{document}